\documentclass{article}

\usepackage{amsmath,amsthm,amssymb,amsfonts}
\usepackage{epsfig,epic}
\usepackage{paralist}
\usepackage{tabularx}
\usepackage{latexsym}
 \usepackage[utf8]{inputenc}
  \usepackage[colorlinks=true]{hyperref}
\hypersetup{urlcolor=blue, citecolor=red}

\newtheorem{ass}{Assumptions}[subsection]

\newtheorem{thm}{Theorem}[subsection]
\newtheorem{lem}[thm]{Lemma}
\newtheorem{rmk}{Remark}

\newtheorem{defin}{Definition}[subsection]
\newtheorem{con}[thm]{Consequences}

\def \eps{\varepsilon}

\begin{document}

\title{Existence, Uniqueness and Positivity of solutions for BGK models for mixtures}
\author{C.~Klingenberg, M:~Pirner}
\date{}

\maketitle

\begin{abstract}
We consider kinetic models for a multi component gas mixture without chemical reactions. 
In the literature\textcolor{black}{,} one can find two types of BGK models in order to describe gas mixtures. One type 
has a sum of BGK type interaction terms in the relaxation operator, for example 
the model described by Klingenberg, Pirner and Puppo \cite{Pirner} which contains well-known models of physicists and engineers for example Hamel \cite{hamel1965} and Gross and Krook \cite{gross_krook1956} as special cases. The other type 
contains only one collision term on the right-hand side, for example 
the well-known model of Andries, Aoki and Perthame \cite{AndriesAokiPerthame2002}.  For each of these two models \cite{Pirner} and \cite{AndriesAokiPerthame2002}\textcolor{black}{,}
we prove existence, uniqueness and positivity of solutions in the first part of the paper. In the second part\textcolor{black}{,} 
we use the first model \cite{Pirner} in order to determine an unknown function in the energy exchange of the macroscopic equations for gas mixtures described by Dellacherie \cite{Dellacherie}. 
\end{abstract}

\section{Introduction}
\label{sec1}
  In this paper\textcolor{black}{,} we shall concern ourselves with a kinetic description of two gases. This is traditionally done via the Boltzmann equation for the two density distributions $f_1$ and $f_2$. Under certain assumptions the complicated interaction terms of the Boltzmann equation can be simplified by a so called BGK approximation, consisting of a collision frequency multiplied by the deviation of the distributions from local Maxwellians. This approximation is constructed in a way such that it  has the same main properties of the Boltzmann equation namely conservation of mass, momentum and energy. In addition\textcolor{black}{,} it has an H-theorem with an entropy inequality leading to an equilibrium which is a Maxwellian.  BGK  models give rise to efficient numerical computations, which are asymptotic preserving, that is they remain efficient even approaching the hydrodynamic regime \cite{Puppo_2007, Jin_2010,Dimarco_2014, Bennoune_2008, Dimarco, Bernard_2015, Crestetto_2012}. The existence and uniqueness of solutions to the BGK equation for one species of gases in bounded domain in space was proven by Perthame and Pulvirenti in \cite{Perthame}.  
 

In this paper\textcolor{black}{,} we are interested in extension\textcolor{black}{s}  of a BGK model to gas mixtures since in applications one often has to deal with mixtures instead of a single gas. \textcolor{black}{From the point of view of physicists, there are a lot of BGK models proposed in the literature concerning gas mixtures. Examples are the model of Gross and Krook in 1956 \cite{gross_krook1956}, the model of Hamel in 1965 \cite{hamel1965}, the model of Garzo, Santos and Brey in 1989 \cite{Garzo1989} and the model of Sofonea and Sekerka in 2001 \cite{Sofonea2001}. They all have one property in common.} Just like the Boltzmann equation for gas mixtures contains a sum of collision terms on the right-hand side, \textcolor{black}{these kind of models also have} a sum of collision terms in the relaxation operator. \textcolor{black}{In 2017 Klingenberg, Pirner and Puppo \cite{Pirner} proposed a kinetic model for gas mixtures which contains these often used models by physicists and engineers as special cases. Moreover, in \cite{Pirner} consistency of this model, like conservation properties, positivity and the H-Theorem, is proven. Since the models from physicists mentioned above are special cases of the model proposed in \cite{Pirner}, consistency of all these models is also proven. Another possible extension to gas mixtures was proposed by Andries, Aoki and Perthame in 2002 \cite{AndriesAokiPerthame2002}. In contrast to the other models it contains only one collision term on the right-hand side. Consistency like conservation properties, positivity and the H-Theorem is also proven there. Brull, Pavan and Schneider proved in \cite{Brull_2012} that the model \cite{AndriesAokiPerthame2002} can be derived by an entropy minimization problem. In recent works, there is the afford to extend this type of BGK model for gas mixtures to gas mixtures with chemical reactions, see for example the model of Bisi and C\'aceras \cite{Bisi}.  } 

\textcolor{black}{To summarize, there are two types of BGK models for gas mixtures in the literature, the model of Andries, Aoki and Perthame \cite{AndriesAokiPerthame2002} and the model of Klingenberg, Pirner and Puppo \cite{Pirner}. The main difference is that \cite{AndriesAokiPerthame2002} contains one relaxation operator on the right-hand side, treating collisions of one species with itself and collisions of one species with the other one in a common relaxation. Whereas the model \cite{Pirner} separates the intra- and interspecies interactions. The motivation of the model \cite{AndriesAokiPerthame2002} was to derive the momentum and energy exchange for the corresponding fluid equations of Maxwellian molecules since in this case it is possible to compute the exchange terms from the Boltzmann equation. The model \cite{Pirner} contains parameters which can be chosen freely. For a special choice of these parameters, they also obtain the exchange terms of Maxwellian molecules, see \cite{Pirner} for details, but for other choices they can obtain different exchange terms. The free parameters can also be used to fix it to data from physical experiments. Numerical simulations of this two models are presented in \cite{Andries} and \cite{Pirner4}. A further issue of kinetic models is to capture the right transport coefficients on the Navier-Sokes level. For the model \cite{Perthame} these coefficients are computed in \cite{Perthame}. For the model \cite{Pirner} this is done in \cite{Pirner5}. Due to the free parameters in this model one has the freedom to choose some transport coefficients such that they fit to experiments. Extensions to an ES-BGK model of the model \cite{Pirner} are also given in \cite{Pirner2}. }

 Our aim is to prove existence, uniqueness and positivity of solutions to the BGK model for mixtures developed in \cite{Pirner} and the model of Andries, Aoki and Perthame in \cite{AndriesAokiPerthame2002}.
This work is largely motivated by \cite{Perthame}  where the global existence of mild solutions of the BGK equation for one species was established, and \cite{Yun} where global existence of mild solutions of the ES-BGK for one species is shown. There is also a result concerning the Boltzmann equation for mixtures in a similar fashion in \cite{Yun2007}.

The outline of the paper is as follows: in subsection\textcolor{black}{,}  we will present the BGK model for two species  developed in \cite{Pirner} and in subsection \ref{sec2.2} the model of Andries, Aoki and Perthame. In subsection \ref{sec3.1}, we prove bounds on the macroscopic quantities which we need in order to show existence and uniqueness of non-negative solutions in section \ref{sec3.2}. In section \ref{sec4}\textcolor{black}{,} we will deduce that all classical solutions with positive initial data remain positive for all later times. In section \ref{sec5}\textcolor{black}{,} we want to use the model from subsection \ref{sec2.1} in order to determine an unknown function in the macroscopic equations for gas mixtures of Dellacherie in \cite{Dellacherie}.


\section{BGK models for mixtures}
\label{sec2}
In this section\textcolor{black}{,} we will present the two types of BGK models for gas mixtures \cite{Pirner} developed by Klingenberg, Pirner and Puppo and \cite{AndriesAokiPerthame2002} by Andries, Aoki and Perthame. For simplicity in the following, we consider a mixture composed  of two different species, but it could be extended to an arbitrary number of species.
\subsection{The BGK approximation for mixtures with two relaxation terms}
\label{sec2.1}
Since we consider a mixture composed of two different species, our kinetic model has two distribution functions $f_1(x,v,t)> 0$ and $f_2(x,v,t) > 0$ where $x\in \mathbb{R}^N$ and $v\in \mathbb{R}^N, N \in \mathbb{N}$ are the phase space variables and $t\geq 0$ the time.  

 Furthermore, for any $f_1,f_2: \Lambda \subset \mathbb{R}^N \times \mathbb{R}^N \times \mathbb{R}^+_0 \rightarrow \mathbb{R}$ with $(1+|v|^2)f_1,(1+|v|^2)f_2 \in L^1(\mathbb{R}^N), f_1,f_2 \geq 0$\textcolor{black}{,} we relate the distribution functions to  macroscopic quantities by mean-values of $f_k$, $k=1,2$
\begin{align}
\int f_k(v) \begin{pmatrix}
1 \\ v  \\ m_k |v-u_k|^2 
\end{pmatrix} 
dv =: \begin{pmatrix}
n_k \\ n_k u_k \\ N n_k T_k 
\end{pmatrix} , \quad k=1,2,
\label{moments}
\end{align} 
where $n_k$ is the number density, $u_k$ the mean velocity and $T_k$ the mean temperature of species \textcolor{black}{ $k$ ($k=1,2$).} Note that in this paper we shall write $T_k$ instead of $k_B T_k$, where $k_B$ is Boltzmann's constant.

The distribution functions are determined by two equations to describe their time evolution. Furthermore\textcolor{black}{,} we only consider binary interactions. So the particles of one species can interact with either themselves or with particles of the other species. \textcolor{black}{We take this into account}  by introducing two interaction terms in both equations. \textcolor{black}{This means that the right-hand side of the equations}  consists of a sum of two collision operator. This \textcolor{black}{structure} is also described in \cite{Cercignani, Cercignani_1975}.
We are interested in a BGK approximation of the interaction terms. This leads us to define \textcolor{black}{two types of} equilibrium distributions. \textcolor{black}{Due to the interaction of a species $k$ with itself, we expect a relaxation towards an equilibrium distribution $M_k$. And due to the interaction of a species with the other one, we expect a relaxation towards a different equilibrium distribution $M_{kj}$}.  Then the model can be written as:
\begin{align} \begin{split} \label{BGK}
\partial_t f_1 + \nabla_x \cdot (v f_1)   &= \nu_{11} n_1 (M_1 - f_1) + \nu_{12} n_2 (M_{12}- f_1),
\\ 
\partial_t f_2 + \nabla_x \cdot (v f_2) &=\nu_{22} n_2 (M_2 - f_2) + \nu_{21} n_1 (M_{21}- f_2), \\
f_1(t=0) &= f_1^0, \\ f_2(t=0)&= f_2^0,
\end{split}
\end{align}
with the Maxwell distributions
\begin{align} 
\begin{split}
M_k(x,v,t) &= \frac{n_k}{\sqrt{2 \pi \frac{T_k}{m_k}}^N }  \exp({- \frac{|v-u_k|^2}{2 \frac{T_k}{m_k}}}), \quad k=1,2,
\\
M_{12}(x,v,t) &= \frac{n_{12}}{\sqrt{2 \pi \frac{T_{12}}{m_1}}^N }  \exp({- \frac{|v-u_{12}|^2}{2 \frac{T_{12}}{m_1}}}),
\\
M_{21}(x,v,t) &= \frac{n_{21}}{\sqrt{2 \pi \frac{T_{21}}{m_2}}^N }  \exp({- \frac{|v-u_{21}|^2}{2 \frac{T_{21}}{m_2}}}).
\end{split}
\label{BGKmix}
\end{align}
Within the next page the unknown variables will be explained. $\nu_{11} n_1$ and $\nu_{22} n_2$ are the collision frequencies of the particles of each species with itself, while $\nu_{12} n_2$ and $\nu_{21} n_1$ are related to interspecies collisions.
To be flexible in choosing the relationship between the collision frequencies, we now assume the relationship
\begin{equation} 
\nu_{12}=\varepsilon \nu_{21}, \quad 0 < \varepsilon \leq 1.
\label{coll}
\end{equation}
The restriction on $\eps$ is without loss of generality. If $\varepsilon >1$, exchange the notation $1$ and $2$ and choose $\frac{1}{\varepsilon}.$ In addition, we assume that all collision frequencies are positive. For the existence and uniqueness proof\textcolor{black}{,} we assume the following restrictions on our collision frequencies
\begin{align}
\begin{split}
\nu_{jk}(x,t) n_k(x,t) = \tilde{\nu}_{jk}\frac{n_k(x,t)}{n_1(x,t)+n_2(x,t)}, \quad j,k=1,2,
\end{split}
\label{cond_coll}
\end{align}
with constants $\tilde{\nu}_{11},\tilde{\nu}_{12},\tilde{\nu}_{21},\tilde{\nu}_{22}>0$. This means that the collision frequencies are given by a constant times the relative density.   \\
The structure of the collision terms ensures that if one collision frequency $\nu_{kl} \rightarrow \infty$ the corresponding distribution function becomes Maxwell distribution. In addition at global equilibrium, the distribution functions become Maxwell distributions with the same velocity and temperature (see section 2.8 in \cite{Pirner}).
The Maxwell distributions $M_1$ and $M_2$ in \eqref{BGKmix} have the same moments as $f_1$ and $f_2$, respectively. With this choice, we guarantee the conservation of mass, momentum and energy in interactions of one species with itself (see section 2.2 in \cite{Pirner}).
The remaining parameters $n_{12}, n_{21}, u_{12}, u_{21}, T_{12}$ and $T_{21}$ will be determined using conservation of \textcolor{black}{the number of particles,} total momentum and \textcolor{black}{energy}, together with some symmetry considerations.
Our model contains three free parameters as will be explained now.
If we assume that \begin{align} n_{12}=n_1 \quad \text{and} \quad n_{21}=n_2,  
\label{density} 
\end{align}
we have conservation of the number of particles, see Theorem 2.1 in \cite{Pirner}.
If we further assume that $u_{12}$ is a linear combination of $u_1$ and $u_2$
 \begin{align}
u_{12}= \delta u_1 + (1- \delta) u_2, \quad \delta \in \mathbb{R},
\label{convexvel}
\end{align} then we have conservation of total momentum
provided that
\begin{align}
u_{21}=u_2 - \frac{m_1}{m_2} \varepsilon (1- \delta ) (u_2 - u_1),
\label{veloc}
\end{align}
see Theorem 2.2 in \cite{Pirner}.
If we further assume that $T_{12}$ is of the following form
\begin{align}
\begin{split}
T_{12} &=  \alpha T_1 + ( 1 - \alpha) T_2 + \gamma |u_1 - u_2 | ^2,  \quad 0 \leq \alpha \leq 1, \gamma \geq 0 ,
\label{contemp}
\end{split}
\end{align}
then we have conservation of total energy
provided that
\begin{align}
\begin{split}
T_{21} =\left[ \frac{1}{N} \varepsilon m_1 (1- \delta) \left( \frac{m_1}{m_2} \varepsilon ( \delta - 1) + \delta +1 \right) - \varepsilon \gamma \right] |u_1 - u_2|^2 \\+ \varepsilon ( 1 - \alpha ) T_1 + ( 1- \varepsilon ( 1 - \alpha)) T_2,
\label{temp}
\end{split}
\end{align}
see Theorem 2.3 in \cite{Pirner}.
In order to ensure the positivity of all temperatures, we need to restrict $\delta$ and $\gamma$ to 
 \begin{align}
0 \leq \gamma  \leq \frac{m_1}{N} (1-\delta) \left[(1 + \frac{m_1}{m_2} \varepsilon ) \delta + 1 - \frac{m_1}{m_2} \varepsilon \right],
 \label{gamma}
 \end{align}
and
\begin{align}
 \frac{ \frac{m_1}{m_2}\varepsilon - 1}{1+\frac{m_1}{m_2}\varepsilon} \leq  \delta \leq 1,
\label{gammapos}
\end{align}
see Theorem 2.5 in \cite{Pirner}. \\ \\ In the following, we want to study the integral version of \eqref{BGK} \textcolor{black}{for $N=3$}.
\begin{defin}\label{milddef}
We call $(f_1, f_2)$ with $(1+|v|^2)f_k \in L^1(\mathbb{R}^N), f_1,f_2 \geq 0$ a mild solution to \eqref{BGK} under the conditions of the collision frequencies \eqref{cond_coll} iff $f_1,f_2$ satisfy
{\footnotesize
\begin{align}
\begin{split}
&f_k(x,v,t)= e^{-\alpha_k(x,v,t)} f_k^0(x-tv,v) \\ &+ e^{-\alpha_k(x,v,t)} \int_0^t [ \tilde{\nu}_{kk} \frac{n_k(x+(s-t)v,s)}{n_k(x+(s-t)v,s)+ n_j(x+(s-t)v,s)} M_k(x+(s-t)v,v,s) \\ &+ \tilde{\nu}_{kj} \frac{n_j(x+(s-t)v,s)}{n_k(x+(s-t)v,s)+ n_j(x+(s-t)v,s)} M_{kj}(x+(s-t)v,v,s)]] e^{\alpha_k(x+(s-t)v,v,s)} ds,
\end{split}
\end{align}}
where $\alpha_k$ is given by
{\small
\begin{align}
\begin{split}
\alpha_k(x,v,t) = \int_0^t [\tilde{\nu}_{kk} \frac{n_k(x+(s-t)v,s)}{n_k(x+(s-t)v,s) + n_j(x+(s-t)v,s) }\\ +\tilde{\nu}_{kj} \frac{n_j(x+(s-t)v,s)}{n_k(x+(s-t)v,s) + n_j(x+(s-t)v,s) } ] ds,
\end{split}
\end{align}}
for $k,j =1,2, ~k\neq j$.
\end{defin}
By construction, a classical solution is always a mild solution. But in order to also allow solutions with a lower regularity, in the following, we want to study existence, uniqueness and positivity of mild solutions.
\subsection{BGK approximation for mixtures with one collision term}
\label{sec2.2}
The next model also describes a gas mixture of Maxwellian molecules, but it contains only one term on the right-hand side \cite{AndriesAokiPerthame2002}.
\begin{align}
\begin{split}
\partial_t f_1 + v \cdot \nabla_x f_1 &= (\nu_{11} n_1 + \nu_{12} n_2)  (M^{(1)} - f_1),
\\
\partial_t f_2 + v \cdot \nabla_x f_2 &= (\nu_{22} n_1 + \nu_{21} n_1)  (M^{(2)} - f_2).
\end{split}
\end{align}
The Maxwell distributions are given by
\begin{align}
\begin{split}
M^{(k)}&=\frac{n_k}{\sqrt{2 \pi \frac{T^{(k)}}{m_k}}^3 } \textcolor{black}{\exp(}- \frac{m_k|v-u^{(k)}|^2}{2 T^{(k)}}), \quad k=1,2,
\end{split}
\end{align}
with the interspecies velocities
\begin{align}
\begin{split}
\\
u^{(k)}&= u_k + 2 \frac{m_j}{m_k + m_j} \frac{\chi_{kj}}{\nu_{kk} n_k + \nu_{kj} n_j} n_j ( u_k - u_j ), \quad k,j=1,2, ~k\neq j,
\end{split}
\end{align}
and the interspecies temperatures
\begin{align}
\begin{split}
T^{(k)} &= T_k - \frac{m_k}{3} |u^{(k)}- u_k|^2 \\ &+ \frac{2}{3} \frac{m_k m_j }{(m_k + m_j)^2} \frac{4 \chi_{kj}}{\nu_{kk} n_k + \nu_{kj} n_j} n_j ( \frac{3}{2}(T_k - T_j ) + m_k \frac{|u_j- u_k|^2}{2}),\\ & \text{ for} \quad k,j=1,2, ~ k\neq j,
\end{split}
\label{TempPerthame}
\end{align}
where $\chi_{12},$ $\chi_{21},$ $\nu_{12}$ and $\nu_{21}$ are parameters which are related to the differential cross section. For the detailed expressions see \cite{AndriesAokiPerthame2002}. We still assume for the existence proof that the collision frequencies have the shape given in \eqref{cond_coll}.
\begin{defin}
We call $(f_1, f_2)$ with $(1+|v|^2)f_k \in L^1(\mathbb{R}^N), f_1,f_2 \geq 0$ a mild solution to \eqref{BGK} under the conditions of the collision frequencies \eqref{cond_coll} iff $f_1,f_2$ satisfy
{\footnotesize
\begin{align}
\begin{split}
&f_k(x,v,t)= e^{-\alpha_k(x,v,t)} f_k^0(x-tv,v) \\ &+ e^{-\alpha_k(x,v,t)} \int_0^t [ \tilde{\nu}_{kk} \frac{n_k(x+(s-t)v,s)}{n_k(x+(s-t)v,s)+ n_j(x+(s-t)v,s)}  \\ &+ \tilde{\nu}_{kj} \frac{n_j(x+(s-t)v,s)}{n_k(x+(s-t)v,s)+ n_j(x+(s-t)v,s)}] M^{(k)}(x+(s-t)v,v,s)] e^{\alpha_k(x+(s-t)v,v,s)} ds
\end{split}
\end{align}}
where $\alpha_k$ is given given as in definition \ref{milddef}, $k,j=1,2, ~ k \neq j$.
\end{defin}
\section{Existence and uniqueness of solutions to the BGK equation for two species}
\label{sec3}
In this section, we start considering several estimates on the macroscopic quantities which we will use in subsection \ref{sec3.2} for the existence and uniqueness of mild solutions. This will be done for the model described in subsection \ref{sec2.1}. The proof for the model presented in subsection \ref{sec2.2} is very similar. \textcolor{black}{S}o we just illustrate  this in remarks.
\subsection{Estimates on the macroscopic quantities}
\label{sec3.1}
First, we present some estimates on macroscopic quantities which we need later for the existence and uniqueness proof. 
\begin{thm}
For any pair of functions $(f_1, f_2)$ with $(1+|v|^2)f_k \in L^1(\mathbb{R}^N),$ $ f_1,f_2 \geq 0$, we define the moments and macroscopic parameters as in \eqref{moments}, \eqref{convexvel}, \eqref{veloc}, \eqref{contemp} and \eqref{temp} and set
\begin{align}
N_q(f_k)= \sup_v |v|^q f_k(v), \quad q\geq0, k=1,2.
\label{Nq}
\end{align}
Then the following estimates hold 
\begin{enumerate}
\item[(i.1)] $\frac{n_k}{T_k^{N/2}} \leq C N_0(f_k) \quad \text{for} \quad k=1,2,$
\item[(i.2)] $\frac{n_1}{T_{12}^{N/2}} \leq C N_0(f_1),$
\item[(1.3)] $\frac{n_2}{T_{21}^{N/2}} \leq C N_0(f_2).$
\end{enumerate}
\end{thm}
\begin{proof}
The proof of $(i.1)$ is exactly the same as the proof of the inequality $(2.2)$ in \cite{Perthame}.
We deduce the estimate $(i.2)$ and $(i.3)$ from $(i.1)$. Furthermore, since we assumed that $f_1, f_2 \geq 0$,$\gamma\geq0$, $0\leq\alpha \leq1$, $\varepsilon\leq1$ and condition \eqref{gamma} both the temperatures $T_1$ and $T_2$ and all coefficients in $T_{12}$ and $T_{21}$ are positive. All in all, this leads to the estimates
{\scriptsize
\begin{align*}
&\frac{n_1}{T_{12}^{N/2}}= \frac{n_1}{(\alpha T_1 +(1-\alpha)T_2 + \gamma |u_1-u_2|^2)^{N/2}} \leq \frac{n_1}{\alpha^{N/2} T_1^{N/2}} \leq C N_0(f_1), \\
& \frac{n_2}{T_{21}^{N/2}}\\= &\frac{n_2}{(\varepsilon(1-\alpha) T_1 +(1-\varepsilon(1-\alpha))T_2 + \left[ \frac{1}{N} \varepsilon m_1 (1- \delta) \left( \frac{m_1}{m_2} \varepsilon ( \delta - 1) + \delta +1 \right) - \varepsilon \gamma \right] |u_1-u_2|^2)^{N/2}}\\ & \leq \frac{n_2}{(1-\varepsilon(1-\alpha))^{N/2} T_2^{N/2}} \leq C N_0(f_{\textcolor{black}{2}}).
\end{align*} }
$\hfill\Box$
\end{proof}
\begin{rmk}
Similar estimates as $(i.2)$ and $(i.3)$ can also be obtained for $T^{(1)},$ $ T^{(2)}$ from \eqref{TempPerthame} in the model presented in subsection \ref{sec2.2} in a\textcolor{black}{n} analogously way if the coefficient in front of $|u_1-u_2|^2$ in \eqref{TempPerthame} is non-negative meaning $\frac{\chi_{12} n_2}{\nu_{11} n_1 + \nu_{12} n_2} \leq 1$ and $\frac{\chi_{21} n_1}{\nu_{22} n_2 + \nu_{21} n_1} \leq 1$. This is reasonable in order to ensure the positivity of the temperatures $T^{(1)}$ and $T^{(2)}$.
\end{rmk}
\begin{thm}
For any pair of functions $(f_1, f_2)$ with $(1+|v|^2)f_k \in L^1(\mathbb{R}^N),$ $ f_1,f_2 \geq 0$, we define the moments as in \eqref{moments}, \eqref{convexvel}, \eqref{veloc}, \eqref{contemp} and \eqref{temp}, then we have
\begin{enumerate}
\item[(ii.1)] $n_k (T_k +|u_k|^2)^{\frac{q-N}{2}} \leq C_q N_q(f_k)$ for $q>N+2$, $k=1,2$,
\item[(ii.2)] $n_1 (T_{12} +|u_{12}|^2)^{\frac{q-N}{2}} \leq C_q (N_q(f_1) + \frac{n_1}{n_2} N_q(f_2))$ for $q>N+2$,
\item[(ii.3)] $n_2 (T_{21} +|u_{21}|^2)^{\frac{q-N}{2}} \leq C_q (\frac{n_2}{n_1} N_q(f_1) + N_q(f_2))$ for $q>N+2$.
\end{enumerate}
\end{thm}
\begin{proof}
The proof of $(ii.1)$ is exactly the same as the proof of the inequality $(2.3)$ in \cite{Perthame}. \\
In order to prove $(ii.2),$ estimate $n_1(T_{12} +|u_{12}|^2)$ using that $f_k\geq 0,$ \eqref{convexvel} and \eqref{contemp} by 
{\small
\begin{align*}
n_1(T_{12} +|u_{12}|^2) &\leq n_1(N T_{12} +|u_{12}|^2) \\&= n_1 (\alpha N T_1 +(1-\alpha) N T_2 + \gamma \textcolor{black}{N} |u_1-u_2|^2 + | \delta u_1 +(1- \delta) u_2|^2 ) \\&= n_1 (\alpha N T_1 +(1-\alpha) N T_2 + (\delta^2 + \textcolor{black}{N}\gamma) |u_1|^2 +((1- \delta)^2 + \gamma) |u_2|^2 \\&+ 2(\delta (1- \delta)- \textcolor{black}{N}\gamma) u_1 \cdot u_2,
\end{align*}}
Using that $|u_1 +u_2|^2 \geq 0$ and $|u_1 -u_2|^2 \geq 0$, we can estimate the term $(\delta (1- \delta)-\textcolor{black}{N}\gamma) u_1 \cdot u_2$ from above by $|\delta(1-\delta)-\textcolor{black}{N}\gamma| \frac{1}{2}(|u_1|^2 +|u_2|^2)$ and obtain
\begin{align*}
n_1(T_{12} +|u_{12}|^2) &\leq n_1 [ \alpha N T_1 +(\delta^2 +\textcolor{black}{N} \gamma +  | \delta(1-\delta) - \textcolor{black}{N}\gamma|) |u_1|^2 +(1- \alpha) N T_2 \\ &+ ((1- \delta)^2 +\textcolor{black}{N} \gamma + |\delta (1- \delta) - \textcolor{black}{N}\gamma|) |u_2|^2]\\ & \leq n_1 [ \max \lbrace \alpha, \delta^2 + \textcolor{black}{N}\gamma +  |\delta(1-\delta)- \textcolor{black}{N}\gamma| \rbrace (NT_1 + |u_1|^2) ] \\&+ \max \lbrace 1- \alpha, ((1- \delta)^2 + \textcolor{black}{N}\gamma +  |\delta (1- \delta)-\textcolor{black}{N} \gamma|) \rbrace (N T_2 + |u_2|^2) ],
\end{align*}
Set $A_1:=  \max \lbrace \alpha, \delta^2 + \textcolor{black}{N}\gamma +  |\delta(1-\delta)- \textcolor{black}{N}\gamma| \rbrace$ and $A_2:= \max \lbrace 1- \alpha, ((1- \delta)^2 +\textcolor{black}{N} \gamma +  |\delta (1- \delta)-\textcolor{black}{N} \gamma|) \rbrace$. Then 
\begin{align*}
n_1(T_{12} +|u_{12}|^2) &\leq n_1 [ A_1 (NT_1 + |u_1|^2 ) + A_2 (N T_2 + |u_2|^2)] \\ &= A_1 \int |v|^2 f_1(v) dv + A_2 \frac{n_1}{n_2} \int |v|^2 f_2(v) dv.
\end{align*}
We split the  integration with respect to the velocity $v$ into $|v|>R_{12}$ and $|v|\leq R_{12}$ for some $R_{12}$ determined later.
 We obtain
\begin{align*}
n_1(T_{12} +|u_{12}|^2) \leq \int_{|v|>R_{12}} \frac{|v|^q}{|v|^{q-2}} (A_1 f_1(v) + A_2 \frac{n_1}{n_2} f_2(v)) dv \\ + \int_{|v| \leq R_{12}} |v|^2 (A_1 f_1(v) + A_2 \frac{n_1}{n_2} f_2(v)) dv.
\end{align*}
Again, since $q> N+2$, we can estimate the integral $\int_{|v|>R_{12}} \frac{1}{|v|^{q-2}} dv$ from above by $C_q R_{12}^{N-q+2}$. In the second integral\textcolor{black}{,} we use that $|v|^2 \leq R_{12}^2$. Then we get
$$ n_1(T_{12} +|u_{12}|^2) \leq C R_{12}^{N-q+2}(A_1 N_q(f_1) + A_2 \frac{n_1}{n_2} N_q(f_2)) + C n_1 R_{12}^2.$$
Now we choose $R_{12}=( \frac{n_1}{A_1 N_q(f_1) + A_2 \frac{n_1}{n_2} N_q(f_2)})^{\frac{1}{N-q}}$ and obtain
$$n_1(T_{12} +|u_{12}|^2) \leq C n_1^{1-\frac{2}{q-N}} (A_1 N_q(f_1) + \frac{n_1}{n_2} A_2 N_q(f_2))^{\frac{2}{q-N}},$$ which is equivalent to the required estimate $(ii.2).$ \\ The proof of $(ii.3)$ is similar to the proof of $(ii.2).$
$\hfill\Box$
\end{proof}
\begin{lem}
For any pair of functions $(f_1, f_2)$ with $(1+|v|^2)f_k \in L^1(\mathbb{R}^N),$ $ f_1,f_2 \geq 0$, we define the moments as in \eqref{moments}, \eqref{convexvel}, \eqref{veloc}, \eqref{contemp} and \eqref{temp}. Let $q \in \mathbb{N}$ or $q-\frac{1}{2} \in \mathbb{N}$, then there exists a constant $A>0$ such that 
$$ |\delta u_1 +(1- \delta) u_2|^q \leq A |u_1|^q + A |u_2|^q,$$
$$ (\alpha T_1 + (1- \alpha) T_2 + \gamma |u_1-u_2|^2 )^q \leq A(T_1^q + T_2^q + |u_1-u_2|^{2q}). $$
\label{Lem}
\end{lem}
This lemma can be proven by induction with respect to $q$.
\begin{thm}
For any pair of functions $(f_1, f_2)$ with $(1+|v|^2)f_k \in L^1(\mathbb{R}^N),$ $ f_1,f_2 \geq 0$, we define the moments as in \eqref{moments}, \eqref{convexvel}, \eqref{veloc}, \eqref{contemp} and \eqref{temp}, then we have
\begin{enumerate}
\item[(iii.1)] $\frac{n_k |u_k|^{N+q}}{[(T_k+|u_k|^2) T_k]^{N/2}} \leq C_q N_q(f_k)$ for any $q>1, k=1,2$,
\item[(iii.2)] $\frac{n_1 |u_{12}|^q}{T_{12}^{N/2}} \leq n_1 C ( \frac{|u_1|^q}{( T_1)^{N/2}} + \frac{|u_2|^q}{(T_2)^{N/2}})$ for any $q>1$,
\item[(iii.3)] $\frac{n_2 |u_{21}|^q}{T_{21}^{N/2}} \leq n_2 C ( \frac{|u_1|^q}{( T_1)^{N/2}} + \frac{|u_2|^q}{(T_2)^{N/2}})$ for any $q>1$.
\end{enumerate}
\end{thm}
\begin{proof}
The proof of (iii.1) is exactly the same as the proof of the inequality (2.3) in \cite{Perthame}.
Estimate $(iii.2)$ is a consequence of Lemma \ref{Lem} using that $\gamma\geq0$, $0\leq\alpha \leq1$ and condition \eqref{gamma}, since we have
\begin{align*}
\frac{n_1|u_{12}|^q}{T_{12}^{N/2}} = \frac{n_1 |\delta u_1 + (1- \delta) u_2|^q}{(\alpha T_1 + (1- \alpha) T_2 + \gamma |u_1-u_2|^2) ^{N/2}} \leq n_1 \frac{A(|u_1|^q +|u_2|^q)}{(\alpha T_1 +(1-\alpha) T_2)^{N/2}} \\ \leq n_1  \frac{ A|u_1|^q}{(\alpha T_1)^{N/2}} + n_1 \frac{A |u_2|^q}{((1-\alpha) T_2)^{N/2}}.
\end{align*}

The proof of $(iii.3)$ is similar to the proof of $(iii.2).$
$\hfill\Box$
\end{proof}
\begin{con}
For any pair of functions $(f_1, f_2)$ with $(1+|v|^2)f_k \in L^1(\mathbb{R}^N), f_1,f_2 \geq 0$, we define the moments as in \eqref{moments}, \eqref{convexvel}, \eqref{veloc}, \eqref{contemp} and \eqref{temp}, then we have
\begin{enumerate}
\item[(iv.1)] $\sup_v |v|^q M_k[f_k] \leq C_q N_q(f_k)$ for $q>N+2$ or $q=0$,
\item[(iv.2)] $\sup_v |v|^q M_{12}[f_1, f_2] \leq C_q ( N_q(f_1) + \frac{n_1}{n_2} N_q(f_2))$ for $q>N+2$ or $q=0$,
\item[(iv.3)] $\sup_v |v|^q M_{21}[f_1, f_2] \leq C_q ( \frac{n_2}{n_1} N_q(f_1) +  N_q(f_2))$ for $q>N+2$ or $q=0$.
\end{enumerate}
\end{con}
\begin{proof}
The proof of $(iv.1)$ is exactly the same as the proof of the inequality $(2.3)$ in \cite{Perthame}. 
Now, the proof of $(iv.2).$ First for $q>N+2$. First\textcolor{black}{,} we compute the maximum of $M_{12}[f_1,f_2]$ and $|v-u_{12}|\textcolor{black}{^q} M_{12}[f_1,f_2]$ similar to the case of one species. \textcolor{black}{The maximum of the Maxwell distribution $M_{12}[f_1,f_2]$ in $v$ is reached when $v=u_{12}$. Therefore $$\max_v M_{12}[f_1,f_2]= \frac{n_1}{(2 \pi \frac{T_{12}}{m_1})^{N/2}}.$$ For the maximum of $|v-u_{12}|^q M_{12}[f_1,f_2],$ we compute the gradient in $v$ and obtain by using product rule
{\footnotesize $$\nabla_v (|v-u_{12}|^q M_{12}[f_1,f_2])=  (v-u_{12}) q |v- u_{12} |^{q-2} M_{12}[f_1,f_2] - \frac{m_1}{T_{12}} |v-u_{12}|^q (v- u_{12}) M_{12}[f_1,f_2].$$} The condition that this expression is equal to zero is equivalent to $$(q (v-u_{12})- \frac{m_1}{T_{12}} |v-u_{12}| (v- u_{12}))=0$$ for $v \neq u_{12}$. We can exclude $v=u_{12}$ since it is a minimum. From this expression\textcolor{black}{,} we can deduce $$ |v-u_{12}|^2 = \frac{T_{12}}{m_1} q.$$ If we insert this into $|v-u_{12}|^q M_{12}[f_1,f_2]$\textcolor{black}{,} we obtain $$ \max_v (|v-u_{12}|^q M_{12}[f_1,f_2]) = \max_v ((\frac{T_{12}}{m_1} q)^{\frac{q}{2}} \frac{n_1}{(2 \pi \frac{T_{12}}{m_1})^{N/2}} e^{- q}).$$   For $|v| \rightarrow \infty$, the expression $|v-u_{12}|^q M_{12}[f_1,f_2]$ tends to zero, so it is equal to the supremum.} All in all, we obtain
\begin{align*}
 \sup_v |v|^q M_{12}[f_1,f_2]\leq \sup_v |v-u_{12}|\textcolor{black}{^q} M_{12}[f_1,f_2] + \sup_v |u_{12}|\textcolor{black}{^q} M_{12}[f_1,f_2] \\ \leq C( n_1 T_{12}^{\frac{q-N}{2}} + n_1 \frac{|u_{1\textcolor{black}{2}}|^q}{T_{12}^{N/2}})= C(n_1(\alpha T_1 + (1- \alpha) T_2 + \gamma |u_1 - u_2|^2 ) + n_1 \frac{|u_{1\textcolor{black}{2}}|^q}{T_{12}^{N/2}}).
 \end{align*}
Since $q-N>0$, we can use Lemma \ref{Lem} in the first term twice and $(iii.2)$ in the second term on the right-hand side and obtain
\begin{align*}
 &\sup_v |v|^q M_{12}[f_1,f_2]\\ &\leq C \left( n_1 (T_1+ |u_1|^2)^{\frac{q-N}{2}} + n_1 \frac{|u_1|^q}{T_1^{N/2}} + \frac{n_1}{n_2} n_2  \left((T_2+ |u_2|^2)^{\frac{q-N}{2}} +  \frac{|u_2|^q}{T_2^{N/2}} \right) \right).
 \end{align*}
The first and the \textcolor{black}{third} term on the right-hand side can be estimated using $(ii.1)$ and the other two terms can be estimated in the same way as in the proof of $(iv.1)$  for one species by $C N_q(f_1)$ and $C \frac{n_1}{n_2} N_q(f_2)$, respectively. Combining both, we get
$$\sup_v |v|^q M_{12}[f_1,f_2] \leq C_q (N_q(f_1) + \frac{n_1}{n_2} N_q(f_2)).$$
For $q=0$\textcolor{black}{,} we use 
$$ \sup_v M_{12}[f_1,f_2] \leq \frac{n_1}{ T_{12}^{N/2}} \leq \frac{n_1}{ T_1^{N/2}} \leq C N_0(f_1),$$
using $(iv.1).$
The proof of $(iv.3)$ is similar to the proof of $(iv.2).$
$\hfill\Box$
\end{proof}
\begin{rmk}
For the multi-species model of Andries, Aoki and Perthame in subsection \ref{sec2.2}\textcolor{black}{,}  we can obtain the same estimates 
{\small
\begin{itemize}
\item[(i.2*/ i.3*)] $\quad\frac{n_k}{(T^{(k)})^{\frac{N}{2}}} \leq C N_0(f_k), \quad k=1,2$,
\item[(ii.2*/ii.3*)] $\quad n_k (T^{(k)} + |u^{(k)}|^2)^{\frac{q-N}{2}} \leq C_q (N_q(f_j)+ \frac{n_j}{n_k} N_q(f_k))$ \text{for} $q>N+2,$ $j \neq k$,
\item[(iii.2*/iii3*)] $ \quad
\frac{n_k |u^{(k)}|^q}{(T^{(k)})^{\frac{N}{2}}} \leq n_k C ( \frac{|u_1|^2}{T_1^{N/2}} + \frac{|u_2|^2}{T_2^{N/2}})$,
\item[(iv.2*/iv.3*)] $\quad\sup_v |v|^q M^{(k)}[f_1,f_2] \leq C_q ( \frac{n_j}{n_k} N_q(f_k) + N_q(f_j))$ \\   for $q>N+2$ or $q=0$, $j \neq k$.
\end{itemize}}
  analogously to the estimates \\ $(i.2/i.3)/(ii.2/ii.3)/(iii.2/iii.3)/(iv.2/iv.3),$ since $u^{(1)},u^{(2)}$ are also linear combinations of $u_1$ and $u_2$ and $T^{(1)}, T^{(2)}$ are also combinations of $T_1,T_2, |u_1-u_2|^2$.
\end{rmk}
\subsection{Existence and uniqueness}
\label{sec3.2}
In this section\textcolor{black}{,} we want to show existence and uniqueness of non-negative solutions in a certain function space using the estimates of the previous section.
For the existence and uniqueness proof, we make the following assumptions:
\begin{ass}
\begin{enumerate}
\item We assume periodic boundary conditions. Equivalently\textcolor{black}{,} we can construct solutions satisfying 
{\small
$$f_k(t,x_1,..., x_N, v_1,...,v_N)= f_k(t,x_1,...,x_{i-1},x_i + a_i,x_{i+1},...x_N,v_1,...v_N)$$}
for all $i=1,...,N$ and a suitable $\lbrace a_i\rbrace \in \mathbb{R}^N$ with positive components, for $k=1,2$.
\item We require that the initial values $f_k^0, k=1,2$ satisfy assumption $1$.
\item We are on the bounded domain in space $\Lambda=\lbrace x \in \mathbb{R}^N | x_i \in (0,a_i)\rbrace$.
\item Suppose that $f_k^0$ satisfies $f_k^0 \geq 0$, $(1+|v|^2) f_k^0 \in L^1(\Lambda \times \mathbb{R}^N)$ with \\$\int f_k^0 dx dv =1, k=1,2$.
\item Suppose $N_q(f_k^0):= \sup f_k^0(x,v)(1+|v|^q) = \frac{1}{2} A_0 < \infty$ for some $q>N+2$.
\item Suppose $\gamma_k(x,t):= \int f_k^0(x-vt,v) dv \geq C_0 >0$ for all $t\in\mathbb{R}.$
\item Assume that the collision frequencies are written as in \eqref{cond_coll} and are positive.
\end{enumerate}
\label{ass}
\end{ass}
With this assumptions\textcolor{black}{,} we can show the following Theorem.
\begin{thm}
Under the assumptions \ref{ass} and the definitions \eqref{moments}, \eqref{convexvel}, \eqref{veloc}, \eqref{contemp} and \eqref{temp}, there exists a unique non-negative mild solution $(f_1,f_2)\in C(\mathbb{R}^+ ; L^1((1+ |v|^2) dv dx)$ of the initial value problem \eqref{BGK}. Moreover, for all $t>0$ the following bounds hold:
\begin{align*}
|u_k(t)|, |u_{12}(t)|, |u_{21}(t)|, T_k(t), T_{12}(t), T_{21}(t), N_q(f_k)(t) &\leq A(t) < \infty, \\
n_k(t) &\geq C_0 e^{-t} >0, \\
T_k(t), T_{12}(t), T_{21}(t) &\geq B(t)>0,
\end{align*} 
for $k=1,2$ and some constants $A(t),B(t)$.
\end{thm}
\label{ex}
\begin{proof}
The idea of the proof is to find a Cauchy sequence of functions in a certain space which converges towards a solution to \eqref{BGK}. The sequence will be constructed in a way such that each member of the sequence satisfies an inhomogeneous transport equation. In this case\textcolor{black}{,} we know results of existence and uniqueness. In order to show that this sequence is a Cauchy sequence\textcolor{black}{,} we need to show that the Maxwellians \textcolor{black}{on} the right-hand side of \eqref{BGK} are Lipschitz continuous with respect to $f_1,f_2$.\\ \\ The proof is structured as follows: First\textcolor{black}{,} we proof some estimates on the macroscopic quantities \eqref{moments}, \eqref{convexvel}, \eqref{veloc}, \eqref{contemp} and \eqref{temp}. From this we can deduce Lipschitz continuity of the Maxwellians $M_1, M_2, M_{12}, M_{21}$ with respect to $f_1$ and $f_2$ which finally leads to the convergence of this Cauchy sequence to a solution to \eqref{BGK}. \\ \\ {\bf Step 1}: Gronwall estimate on $N_q(f_k(t))$ given by \eqref{Nq} \\ \\
If $f_1$ is a mild solution according to definition \ref{milddef}, we have
{\footnotesize
\begin{align*}
&N_q(f_1)= \sup_v |v|^q f_1 \leq  e^{-\alpha_1(x,v,t)} \sup_v|v|^q f_1^0(x-tv,v) \\ &+ \sup_v |v|^q [ e^{-\alpha_1(x,v,t)} \int_0^t [ \tilde{\nu}_{11} \frac{n_1(x+(s-t)v,s)}{n_1(x+(s-t)v,s)+ n_2(x+(s-t)v,s)} M_1(x+(s-t)v,v,s) \\ &+ \tilde{\nu}_{12} \frac{n_2(x+(s-t)v,s)}{n_1(x+(s-t)v,s)+ n_2(x+(s-t)v,s)} M_{12}(x+(s-t)v,v,s)] e^{\alpha_1(x+(s-t)v,v,s)} ds ].
\end{align*}}
Since $\alpha_1$ is non-negative, we can estimate $e^{- \alpha_1(x,v,t)}$ in front of the initial data from above by $1$. Since we assumed that the collision frequencies have the shape given in \eqref{cond_coll}, we can estimate the integrand in the exponential function $e^{ \alpha_1(x,v,t)} e^{\alpha_1(x+ (s-t)v,v,s)}$ by a constant and obtain
{\footnotesize
\begin{align*}
&N_q(f_1)= \sup_v |v|^q f_1 \leq  \sup_v|v|^q f_1^0(x-tv,v) \\ &+ \sup_v |v|^q [ \int_0^t e^{-C(t-s)} [ C \frac{n_1(x+(s-t)v,s)}{n_1(x+(s-t)v,s)+ n_2(x+(s-t)v,s)} M_1(x+(s-t)v,v,s) \\ &+ C \frac{n_2(x+(s-t)v,s)}{n_1(x+(s-t)v,s)+ n_2(x+(s-t)v,s)} M_{12}(x+(s-t)v,v,s)]]  ds ].
\end{align*}}
Using assumption 5 (in \textcolor{black}{the assumptions} \ref{ass}) \textcolor{black}{and the fact that},  we can estimate $e^{-C(t-s)}$ from above by $1$ since $s$ is between $0$ and $t$\textcolor{black}{, we get} \begin{align*}
&N_q(f_1)= \sup_v |v|^q f_1 \leq  \frac{1}{2} A_0 +   \int_0^t C \sup_x [  \frac{n_1(x,s)}{n_1(x,s)+ n_2(x,s)} \sup_{v} |v|^q M_1(x,v,s) \\ &+  \frac{n_2(x,s)}{n_1(x,s)+ n_2(x,s)} \sup_{v} |v|^q M_{12}(x,v,s)]  ds ].
\end{align*}
With $(iv.1)$ and $(iv.2),$ we obtain
{\footnotesize
\begin{align*}
N_q(f_1)&= \sup_{x,v} |v|^q f_1 \\ &\leq  \frac{1}{2} A_0 +   \int_0^t C_q \sup_x [\frac{n_1(x,s)+n_2(x,t)}{n_1(x,s)+ n_1(x,s)}  N_q(f_1)(s) +  \frac{n_1(x,s)}{n_1(x,s)+ n_2(x,s)} N_q(f_2(s))]  ds \\ &\leq \frac{1}{2} A_0 +   \int_0^t C_q [  \sup_x N_q(f_1)(s) +  \sup_x N_q(f_2)(s)]  ds.
\end{align*}}
Similarly, we can estimate $N_q(f_2)$ by
\begin{align*}
N_q(f_2)&= \sup_v |v|^q f_2 \leq  \frac{1}{2} A_0 +   \int_0^t C_q [ \sup_x N_q(f_1)(s) +  \sup_x N_q(f_2)(s)]  ds.
\end{align*}
We add both inequalities and obtain
$$ N_q(f_1)+N_q(f_2) \leq A_0 + \int_0^t C_q [\sup_x N_q(f_1)(s)+ \sup_x N_q(f_2)(s)] ds.$$
With Gronwalls Lemma, we obtain
\begin{align}
N_q(f_1)(t) + N_q(f_2)(t) \leq A_0 e^{C_q t} \quad \text{ for} \quad q>N+2 \quad \text{ or} \quad q=0.
\label{G}
\end{align}
\\ \\ {\bf Step 2}: Estimate on the densities \\ \\
If $f_k\geq0$ is a solution, it satisfies
\begin{align*}
\partial_t f_k + v \cdot \nabla_x f_k  &= \tilde{\nu}_{kk} \frac{n_k}{n_k + n_j} (M_k - f_k) + \tilde{\nu}_{kj} \frac{n_j}{n_k+n_j} (M_{kj} - f_k)\\ &\geq -(\tilde{\nu}_{kk} + \tilde{\nu}_{kj}) f_k,
\end{align*}
If we write this in the mild formulation, this leads to
$$f_k(t,x,v) \geq e^{-(\tilde{\nu}_{kk} + \tilde{\nu}_{kj})t} f_k^0(x-tv).$$ Integrating this with respect to $v$ leads with assumption $6$ (in assumptions \ref{ass}) to the estimate of the densities
\begin{align}
\begin{split}
n_k(x,t) &\geq e^{-(\tilde{\nu}_{kk} + \tilde{\nu}_{kj})t} \int f_k^0(x-vt,v) dv \\ &\geq e^{-(\tilde{\nu}_{kk} + \tilde{\nu}_{kj})t} \gamma_k(x,t) \geq e^{-(\tilde{\nu}_{kk} + \tilde{\nu}_{kj})t} C_0 >0.
\end{split}
\label{D}
\end{align}
{\bf Step 3}: Estimate on the temperatures \\ \\
Now, we estimate the temperatures from below. First, we consider $T_k^{N/2}$. We can estimate it from below using $(i.1)$
$$T_k^{N/2}(t) \geq \frac{C n_k(t)}{N_0(f_k(t))}. $$ Using \eqref{G} and \eqref{D}, we obtain
$$ T_k^{N/2}(t) \geq \frac{C e^{-(\tilde{\nu}_{kk} + \tilde{\nu}_{kj})t} C_0}{A e^{C_qt}}=:B(t)>0.$$ We obtain the same estimate for $T_{12}^{N/2}$ using $(i.2),$ \eqref{G} and \eqref{D}, and for $T_{21}^{N/2}$ using $(i.3),$ \eqref{G} and \eqref{D}. \\ \\
{\bf Step 4}: Estimates on the velocities \\ \\
We estimate
$T_k+ |u_k|^2,$ $T_{12}+ |u_{12}|^2,$ and $T_{21}+|u_{21}|^2$  first using $(ii.1), (ii.2)$ and $(iii.3)$, respectively  and then using \eqref{G} and \eqref{D}. For example
$$T_{12}+ |u_{12}|^2\leq \frac{C_q(N_q(f_1) + \frac{n_1}{n_2}N_q(f_2))^{\frac{2}{q-N}}}{n_1} \leq \frac{C_q A e^{C_q \frac{2}{q-N}t}}{e^{-\frac{2}{q-N} t} C_0^{\frac{2}{q-N}}} < A(t)< \infty.$$
\\ \\ {\bf Step 5}: Lipschitz continuity \\ \\
The next step of the proof is to show Lipschitz continuity of the operators $f_k \mapsto M_k[f_k]$, $(f_1,f_2) \mapsto \frac{n_2}{n_1+n_2} M_{12}[f_1,f_2]$ and $(f_1,f_2) \mapsto \frac{n_1}{n_1 + n_2} M_{21}[f_1,f_2]$, when $(f_1,f_2)$ are restricted to
{\footnotesize
\begin{align}
\Omega=\lbrace (f_1,f_2)\in L^1(\Lambda \times \mathbb{R}^N; (1+|v|^2) dv dx) | f_k \geq 0, N_q(f_k)<A, \min (n_k,T_k)>C,k=1,2 \rbrace.
\label{Omega}
\end{align}
}
The proof for $f_k \mapsto M_k[f_k]$ is given in \cite{Perthame}. So it remains to show Lipschitz continuity for $(f_1,f_2) \mapsto \frac{n_2}{n_1+n_2} M_{12}[f_1,f_2]$ and for $(f_1,f_2) \mapsto \frac{n_1}{n_1+n_2} M_{21}[f_1,f_2]$. We only prove the first case since the second one is similar to one the first one. 
For any pair $(f_1^i,f_2^i), i=1,2$ in the subset $\Omega$, define $(n_1^i, u_{12}^i, T_{12}^i)$ as their corresponding moments. Set $$(n_1^{\Theta}, n_2^{\Theta},u_{12}^{\Theta}, T_{12}^{\Theta}) = \Theta (n_1^1, n_2^1, u_{12}^1, T_{12}^1) + (1- \Theta)(n_1^2, n_2^2, u_{12}^2, T_{12}^2),$$ and $$M_{12}(\Theta)= \frac{n_1^{\Theta}}{(2\pi T_{12}^{\Theta}/m_1)^{N/2}} e^{-\frac{|v-u_{12}^{\Theta}|^2}{2T_{12}^{\Theta}/m_1}} \frac{n_2^{\Theta}}{n_1^{\Theta}+n_2^{\Theta}}.$$
Then we have
\begin{align*}
\int | \frac{n_2^1}{n_1^1+n_2^1} M_{12}[f_1^1,f_2^1]- \frac{n_2^2}{n_1^2+n_2^2} M_{12}[f_1^2,f_2^2]|(1+|v|^2)dv \\ = \int |M_{12}(1) - M_{12}(0)|(1+|v|^2) dv.
\end{align*}
Now, we use the Taylor formula with first derivative as remainder and the chain rule and obtain
\begin{align*}
\int | \frac{n_2^1}{n_1^1+n_2^1} M_{12}[f_1^1,f_2^1]- \frac{n_2^2}{n_1^2+n_2^2} M_{12}[f_1^2,f_2^2]|(1+|v|^2)dv \\= \int |\frac{\partial M_{12}}{\partial \Theta}(\Theta)|(1+|v|^2) dv \\\leq \int_0^1 \int (| \frac{\partial M_{12}}{\partial n_1^{\Theta}}(\Theta) \frac{\partial n_1^{\Theta}}{\partial \Theta}| + | \frac{\partial M_{12}}{\partial u_{12}^{\Theta}}(\Theta) \frac{\partial u_{12}^{\Theta}}{\partial \Theta}|+| \frac{\partial M_{12}}{\partial T_{12}^{\Theta}}(\Theta) \frac{\partial T_{12}^{\Theta}}{\partial \Theta}|\\+| \frac{\partial M_{12}}{\partial n_2^{\Theta}}(\Theta) \frac{\partial n_2^{\Theta}}{\partial \Theta}|(1+|v|^2)dv d\Theta \\= \int_0^1 \int (| \frac{\partial M_{12}}{\partial n_1^{\Theta}}(\Theta) ||n_1^1-n_1^2|  + | \frac{\partial M_{12}}{\partial u_{12}^{\Theta}}(\Theta) ||u_{12}^1-u_{12}^2|\\+| \frac{\partial M_{12}}{\partial T_{12}^{\Theta}}(\Theta) ||T_{12}^1-T_{12}^2|+| \frac{\partial M_{12}}{\partial n_2^{\Theta}}(\Theta) ||n_2^1-n_2^2|(1+|v|^2)dv d\Theta.
\end{align*}
An explicit calculation of the derivatives leads to
\begin{align*}
\int | \frac{n_2^1}{n_1^1+n_2^1} M_{12}[f_1^1,f_2^1]- \frac{n_2^2}{n_1^1+n_2^2} M_{12}[f_1^2,f_2^2]|(1+|v|^2)dv \\ \leq\int_0^1  ((1+|u_{12}^{\Theta}|^2 + N T_{12}^{\Theta})|n_1^1-n_1^2|  \\ + C[\frac{n_2^{\Theta}}{n_1^{\Theta}+n_2^{\Theta}} \frac{n_1^{\Theta}}{(T_{12}^{\Theta})^{1/2}}(1+ |u_{12}^{\Theta}|^2+T_{12}^{\Theta}) ] |u_{12}^1-u_{12}^2|\\+C[\frac{n_2^{\Theta}}{n_1^{\Theta}+n_2^{\Theta}} \frac{n_1^{\Theta}}{T_{12}^{\Theta}}(1+ |u_{12}^{\Theta}|^2+T_{12}^{\Theta}) ]|T_{12}^1-T_{12}^2|\\+(1+|u_{12}^{\Theta}|^2 + N T_{12}^{\Theta})|n_2^1-n_2^2|d\Theta.
\end{align*}
The main difference to the one species case is the additional term $| \frac{\partial M_{12}}{\partial n_2}(\Theta) |$ and the term $\partial_{n_1^{\Theta}}(\frac{n_1^{\Theta} n_2^{\Theta}}{n_1^{\Theta}+n_2^{\Theta}})$. For the second term\textcolor{black}{,} we computed $\partial_{n_1^{\Theta}}(\frac{n_1^{\Theta} n_2^{\Theta}}{n_1^{\Theta}+n_2^{\Theta}}) = \frac{n_2^{\Theta}}{n_1^{\Theta}+n_2^{\Theta}} - \frac{n_1^{\Theta} n_2^{\Theta}}{(n_1^{\Theta}+ n_2^{\Theta})^2}$ which we can estimate from above by $\frac{n_2^{\Theta}}{n_1^{\Theta}+n_2^{\Theta}} \leq 1$. All terms in front of the norms $|\cdot|$ are bounded by a constant due to the estimate on the temperature $T_{12}^{N/2}$ and the estimate on $T_{12} + |u_{12}|^2$ proven in step $2$ and $3$. Furthermore, we can estimate
$$|n_1^1-n_1^2| \leq \int (1+|v|^2)|f_1^1-f_1^2|dv,$$ and
$$|n_2^1-n_2^2| \leq \int (1+|v|^2)|f_2^1-f_2^2|dv,$$
\begin{align*}
U:=\frac{n_1^{\Theta} n_2^{\Theta}}{n_1^{\Theta}+ n_2^{\Theta}} |u_{12}^1-u_{12}^2| = \frac{n_1^{\Theta} n_2^{\Theta}}{n_1^{\Theta}+ n_2^{\Theta}} |\alpha u_1^1 +(1-\alpha) u_2^1 - \alpha u_1^2 - (1- \alpha)u_2^2| \\ \leq \frac{n_2^{\Theta} (n_1^1 + n_1^2)}{n_1^{\Theta} + n_2^{\Theta}} \alpha |u_1^1 - u_1^2| + \frac{n_1^{\Theta}(n_2^1+n_2^2)}{n_1^{\Theta}+n_2^{\Theta}}(1-\alpha) |u_2^1-u_2^2|.
\end{align*}
Since $\frac{n_2^{\Theta}}{n_1^{\Theta} + n_2^{\Theta}}$ and $\frac{n_1^{\Theta}}{n_1^{\Theta} + n_2^{\Theta}}$ are smaller or equal $1$, we can estimate
{\footnotesize
\begin{align*}
&U
\leq  (n_1^1 + n_1^2) \alpha |u_1^1 - u_1^2| + (n_2^1+n_2^2)(1-\alpha) |u_2^1-u_2^2| \\
&\leq \alpha |n_1^1 u_1^1 - n_1^1 u_1^2 + n_1^2 u_1^1 - n_1^2 u_1^2|+(1-\alpha) |n_2^1 u_2^1 - n_2^1 u_2^2 + n_2^2 u_2^1 - n_2^2 u_2^2|\\ 
&\leq \alpha |n_1^1 u_1^1 - n_1^1 u_1^2|+\alpha|n_1^2 u_1^1 - n_1^2 u_1^2|+(1-\alpha) |n_2^1 u_2^1 - n_2^1 u_2^2|+\textcolor{black}{(1-\alpha)}|n_2^2 u_2^1 - n_2^2 u_2^2|\\
&\leq \alpha |n_1^1 u_1^1 - n_1^2 u_1^2 + n_1^2 u_1^2 - n_1^1 u_1^2|+ \alpha |n_1^1 u_1^1 - n_1^2 u_1^2 + n_1^2 u_1^1 - n_1^1 u_1^1|\\ 
&+(1-\alpha) |n_2^1 u_2^1 - n_2^2 u_2^2 + n_2^2 u_2^2 - n_2^1 u_2^2|+ (1-\alpha) |n_2^1 u_2^1 - n_2^2 u_2^2 + n_2^2 u_2^1 - n_2^1 u_2^1|\\
& \leq \alpha [|n_1^1 u_1^1 - n_1^2 u_1^2| + |u_1^2| |n_1^2 - n_1^1| + |n_1^1 u_1^1 - n_1^2 u_1^2| + |u_1^1| |n_1^2 - n_1^1|]\\&+(1-\alpha) [|n_2^1 u_2^1 - n_2^2 u_2^2| + |u_2^2| |n_2^2 - n_2^1| + |n_2^1 u_2^1 - n_2^2 u_2^2| + |u_2^1| |n_2^2 - n_2^1|].
\end{align*}}
Due to the previous estimates on the velocities in step $4$, the velocities are bounded and therefore
$$ U \leq C [ \int \textcolor{black}{(1+|v|)^2} |f_1^1-f_1^2|dv+ \int (1+|v|^2) |f_2^1 - f_2^2| dv]. $$
In an analogous way, we can estimate 
\begin{align*}
 \frac{n_1^{\Theta} n_2^{\Theta}}{n_1^{\Theta}+ n_2^{\Theta}} |T_{12}^1-T_{12}^2| \leq C [ \int \textcolor{black}{(1+|v|)^2} |f_1^1-f_1^2|dv+ \int (1+|v|^2) |f_2^1 - f_2^2| dv].
 \end{align*}
This all combines to the desired Lipschitz estimate.
\\ \\ {\bf Step 6}: Existence and Uniqueness of non-negative solutions in $\bar{\Omega}$ (see definition of $\Omega$ in \eqref{Omega})\\ \\
Now, introduce the sequence  $\lbrace(f_1^n,f_2^n)\rbrace$ of mild solutions to
\begin{align} \begin{split} 
\partial_t f_1^n + \textcolor{black}{v \cdot \nabla_x  f_1^n}   &= \tilde{\nu}_{11} \frac{n_1^{n-1}}{n_1^{n-1}+n_2^{n-1}} (M_1[f_1^{n-1}\textcolor{black}{]} - f_1^n) \\ &+ \tilde{\nu}_{12} \frac{n_2^{n-1}}{n_1^{n-1}+n_2^{n-1}} (M_{12}[f_1^{n-1},f_2^{n-1}]- f_1^n),
\\ 
\partial_t f_2^n + \textcolor{black}{v \cdot \nabla_x  f_2^n}   &= \tilde{\nu}_{22} \frac{n_2^{n-1}}{n_1^{n-1}+n_2^{n-1}} (M_2[f_2^{n-1}\textcolor{black}{]} - f_2^n) \\&+ \tilde{\nu}_{21} \frac{n_1^{n-1}}{n_1^{n-1}+n_2^{n-1}} (M_{21}[f_1^{n-1},f_2^{n-1}]- f_2^n), \\ f_1^0&= f_1(t=0), \\ f_2^0&=f_2(t=0).
\end{split}
\end{align}
Since the zeroth functions are known as the initial values, these are inhomogeneous transport equations for fixed $n\in \mathbb{N}$. For an inhomogeneous transport equation\textcolor{black}{,} we know the existence of a unique mild solution in the periodic setting
{\scriptsize
\begin{align}
\begin{split}
&f_1^n(x,v,t)= e^{-\alpha_1^{n-1}(x,v,t)} f_1^0(x-tv,v) \\ &+ e^{-\alpha_1^{n-1}(x,v,t)} \int_0^t [ \tilde{\nu}_{11} \frac{n_1^{n-1}(x+(s-t)v,s)}{n_1^{n-1}(x+(s-t)v,s)+ n_2^{n-1}(x+(s-t)v,s)} M_1^{n-1}(x+(s-t)v,v,s) \\ &+ \tilde{\nu}_{12} \frac{n_2^{n-1}(x+(s-t)v,s)}{n_1^{n-1}(x+(s-t)v,s)+ n_2^{n-1}(x+(s-t)v,s)} M_{12}^{n-1}(x+(s-t)v,v,s)] e^{\alpha_1^{n-1}(x+(s-t)v,v,s)} ds
\end{split}
\\
\begin{split}
&f_2^n(x,v,t)= e^{-\alpha_2^{n-1}(x,v,t)} f_2^0(x-tv,v) \\ &+ e^{-\alpha_2^{n-1}(x,v,t)} \int_0^t [ \tilde{\nu}_{22} \frac{n_2^{n-1}(x+(s-t)v,s)}{n_1^{n-1}(x+(s-t)v,s)+ n_2^{n-1}(x+(s-t)v,s)} M_2^{n-1}(x+(s-t)v,v,s) \\ &+ \tilde{\nu}_{21} \frac{n_1^{n-1}(x+(s-t)v,s)}{n_1^{n-1}(x+(s-t)v,s)+ n_2^{n-1}(x+(s-t)v,s)} M_{21}^{n-1}(x+(s-t)v,v,s)] e^{\alpha_2^{n-1}(x+(s-t)v,v,s)} ds
\end{split}
\end{align}}
Now, we show that $\lbrace(f_1^n,f_2^n)\rbrace$ is a Cauchy sequence in $\Omega$. Then, since $\bar{\Omega}$ is complete, we can conclude convergence in $\bar{\Omega}$. First, we show that $\lbrace(f_1^n,f_2^n)\rbrace$ is in $\Omega$.
\begin{itemize}
\item $f_1^n,f_2^n$ are in $L^1((1+|v|^2)dv dx)$ since $f_1^0,f_2^0$ are in $L^1((1+|v|^2)dv dx)$.
\item $f_1^n,f_2^n\geq 0$ since $f_1^0,f_2^0\geq 0$.
\item $N_q(f_k^n) < A$, $\min(n_k^n,T_k^n)>C$, since all estimates in step $1,2$ and $4$ are independent of $n$.
\end{itemize}
Now, $\lbrace(f^n_1,f^n_2)\rbrace$ is a Cauchy sequence in $\Omega$ since we have
{\scriptsize
\begin{align*}
&||f_1^n-f_1^{n-1}||_{L^1((1+|v|^2)dv dx)}\\ &\leq \int_{\Lambda} \int_{\mathbb{R}^n} e^{-\alpha_1^{n-1}(x,v,t)} \int_0^t e^{\alpha_1^{n-1}(x+(s-t)v,v,s)} |\tilde{\nu}_{11}^{n-1} \frac{n_1^{n-1}(x+(s-t)v,s)}{n_1^{n-1}(x+(s-t)v,s)+ n_2^{n-1}(x+(s-t)v,s)}\\ &M_1^{n-1}(x+(s-t)v,v,s) - \tilde{\nu}_{11}^{n-2} \frac{n_1^{n-2}(x+(s-t)v,s)}{n_1^{n-2}(x+(s-t)v,s)+ n_2^{n-2}(x+(s-t)v,s)} \\&M_1^{n-2}(x+(s-t)v,v,s)|ds (1+|v|^2) dx dv \\&+ \int_{\Lambda} \int_{\mathbb{R}^n} e^{-\alpha_1^{n-1}(x,v,t)} \int_0^t e^{\alpha_1^{n-1}(x+(s-t)v,v,s)} |\tilde{\nu}_{12}^{n-1} \frac{n_2^{n-1}(x+(s-t)v,s)}{n_1^{n-1}(x+(s-t)v,s)+ n_2^{n-1}(x+(s-t)v,s)}\\ &M_{12}^{n-1}(x+(s-t)v,v,s) - \tilde{\nu}_{12}^{n-2} \frac{n_2^{n-2}(x+(s-t)v,s)}{n_1^{n-2}(x+(s-t)v,s)+ n_2^{n-2}(x+(s-t)v,s)} \\&M_{12}^{n-2}(x+(s-t)v,v,s)|ds (1+|v|^2) dx dv.
\end{align*}}
Now\textcolor{black}{,} we use the Lipschitz continuity of the Maxwellians
{\scriptsize
\begin{align*}
&||f_1^n-f_1^{n-1}||_{L^1((1+|v|^2)dv dx)}\\ &\leq C \int_{\Lambda} \int_{\mathbb{R}^n} e^{-\alpha_1^{n-1}(x,v,t)} \int_0^t e^{\alpha_1^{n-1}(x+(s-t)v,v,s)} | f_1^{n-1}(x+(s-t)v,v,s) \\&- f_1^{n-2}(x+(s-t)v,v,s)|ds (1+|v|^2) dx dv \\&+ \int_{\Lambda} \int_{\mathbb{R}^n} e^{-\alpha_1^{n-1}(x,v,t)} \int_0^t e^{\alpha_1^{n-1}(x+(s-t)v,v,s)} [ | f_{1}^{n-1}(x+(s-t)v,v,s) - f_{1}^{n-2}(x+(s-t)v,v,s)|\\&+| f_{2}^{n-1}(x+(s-t)v,v,s) - f_{2}^{n-2}(x+(s-t)v,v,s)|] ds (1+|v|^2) dx dv
\\ &\leq   e^{-Ct} \int_0^t e^{Cs} [ || f_{1}^{n-1}(s) - f_{1}^{n-2}(s)||_{L^1((1+|v|^2)dv dx}+|| f_{2}^{n-1}(s) - f_{2}^{n-2}(s)||_{L^1((1+|v|^2)dv dx}] ds.
\end{align*}}
Similarly, we get for species $2$
{\normalsize
\begin{align*}
||f_2^n-f_2^{n-1}||_{L^1((1+|v|^2)dv dx)} \leq   e^{-Ct} \int_0^t e^{Cs} [ || f_{1}^{n-1}(s) - f_{1}^{n-2}(s)||_{L^1((1+|v|^2)dv dx}\\+|| f_{2}^{n-1}(s) - f_{2}^{n-2}(s)||_{L^1((1+|v|^2)dv dx}] ds.
\end{align*}}
Doing this inductively, we obtain
{\footnotesize
\begin{align*}
||f_1^n-f_1^{n-1}||_{L^1((1+|v|^2)dv dx)}\\  \leq   (e^{-Ct})^n \int_0^t \cdots \int_0^t e^{Cs_1} \cdots e^{Cs_n} [ || f_{1}^{1}(s_n) - f_{1}^{0}||_{L^1((1+|v|^2)dv dx} \\+|| f_{2}^{1}(s_n) - f_{2}^{0}||_{L^1((1+|v|^2)dv dx}] ds_1 \cdots ds_n \\ \leq \frac{1}{C^n}(1-e^{-Ct})^n  [\sup_{0\leq s \leq t} || f_{1}^{1}(s) - f_{1}^{0}||_{L^1((1+|v|^2)dv dx}+\sup_{0\leq s \leq t}|| f_{2}^{1}(s) - f_{2}^{0}||_{L^1((1+|v|^2)dv dx}],
\end{align*}}
with a constant $C>1$.
So, for species one, we obtain
{\footnotesize
\begin{align*}
\sup_{0\leq t \leq T}||f_1^{n+m}-f_1^{n}||_{L^1((1+|v|^2)dv dx)}\\ \leq \sup_{0\leq t \leq T} [ ||f_1^{n+m}-f_1^{n+m-1}||_{L^1((1+|v|^2)dv dx)} + \cdots + ||f_1^{n+1}-f_1^{n}||_{L^1((1+|v|^2)dv dx)} \\ \leq \sup_{0\leq t \leq  T}((\frac{1}{C}(1-e^{-Ct}))^{n+m} + \cdots + (\frac{1}{C}(1-e^{-t}))^n) [\sup_{0 \leq s \leq t}||f_1^1(s)-f_1^{0}||_{L^1((1+|v|^2)dv dx)} \\ + \sup_{0 \leq s \leq t} ||f_2^1(s)-f_2^{0}||_{L^1((1+|v|^2)dv dx)}]
\\ \leq ((C(T))^{n+m} + \cdots + C(T)^n) [\sup_{0 \leq s \leq T}||f_1^1(s)-f_1^{0}||_{L^1((1+|v|^2)dv dx)} \\ + \sup_{0 \leq s \leq T} ||f_2^1(s)-f_2^{0}||_{L^1((1+|v|^2)dv dx)}]
\\ \leq C(T)^n \sum_{j=1}^{\infty} (C(T))^j [\sup_{0 \leq s \leq T}||f_1^1(s)-f_1^{0}||_{L^1((1+|v|^2)dv dx)} \\ + \sup_{0 \leq s \leq T} ||f_2^1(s)-f_2^{0}||_{L^1((1+|v|^2)dv dx)}]
\\ \leq \frac{C(T)^n}{1-C(T)} [\sup_{0 \leq s \leq T}||f_1^1(s)-f_1^{0}||_{L^1((1+|v|^2)dv dx)} \\ + \sup_{0 \leq s \leq T} ||f_2^1(s)-f_2^{0}||_{L^1((1+|v|^2)dv dx)}],
\end{align*}}
which converges to zero as $n \rightarrow \infty$ since $C(T)=\frac{1- e^{-CT}}{C}<1$.
In order to prove that the limit is a mild solution to \eqref{BGK} and the uniqueness of solutions to \eqref{BGK}\textcolor{black}{,} we use standard arguments similar as in the proof of the fix point Theorem of Banach and the Theorem of Picard- Lindel\"of.
$\hfill\Box$
\end{proof}
\begin{rmk}
$M^{(1)}$ and $M^{(2)}$ in the model of Andries, Aoki and Perthame in \cite{AndriesAokiPerthame2002} and subsection \ref{sec2.2} have the same structure as $M_{12}$ and $M_{21}$, respectively, meaning that the velocities $u^{(1)}$ and $u^{(2)}$ and the temperatures $T^{(1)}$ and $T^{(2)}$ of $M^{(1)}$ and $M^{(2)}$, respectively  have the same structure as the velocities $u_{12}$ and $u_{21}$ and the temperatures $T_{12}$ and $T_{21}$. So the proof of Theorem 3.2.1 for the model in subsection \ref{sec2.2} and \cite{AndriesAokiPerthame2002} goes through analogously as for the model in subsection \ref{sec2.1} and \cite{Pirner}.
\end{rmk}
\section{Positivity of solutions of the BGK approximation for two species}
\label{sec4}
In this section\textcolor{black}{,} we want to show that every classical solution with positive initial data remains positive.
\subsection{Idea of the proof}
\label{sec4.1}
Our aim is to prove that all classical solutions to \eqref{BGK} - \eqref{temp} under the assumptions \ref{ass} with positive initial data are positive for all larger times $t>0$. The idea of the proof is as follows. In the previous section, we stated our result about existence and uniqueness of non-negative solutions.
Then, with a Gronwall estimate on the densities, we deduce that this non-negative solution can be estimated from below by an exponential function. Considering the solution along characteristics\textcolor{black}{,} we will see that when the densities are positive the solution is also positive. With this and continuity in time, we can conclude that for positive initial data there cannot be a solution which becomes zero or negative  at a time $t>0$. So all classical solutions to \eqref{BGK} - \eqref{temp} are positive.

\subsection{Estimate on the densities}
\label{sec4.2}
\begin{lem} \label{est}
If $f_k \geq 0$ is a mild solution to \eqref{BGK} - \eqref{temp} and  
\begin{align*}
\gamma_k(x,t):= \int f_k^0(x-vt,v) dv \geq C_0 > 0,
\end{align*} 
for all $t\geq 0$, $k=1,2$, then the densities satisfy the estimate $$n_k(x,t) \geq C_0 e^{-(\tilde{\nu}_{kk} + \tilde{\nu}_{kj})t},$$ for all $t\geq 0$ where $C_0>0$ is a positive constant. 
\end{lem}
\begin{proof}
See step 1 in the proof of Theorem 2.3.1.
$\hfill\Box$
\end{proof}

\subsection{Positivity of non-negative solutions}
\label{sec4.3}
\begin{lem}[Positivity of non-negative solutions]
Let $(f_1,f_2)$ with \\$f_1, f_2 \geq 0 $ be a mild solution to \eqref{BGK}-\eqref{temp} with positive initial data under the assumptions \ref{ass}. Then $f_1,f_2$ are even positive, that means $f_1,f_2 >0$ a.e.
\end{lem}
\begin{proof}
We prove the statement for $f_1$, the proof for $f_2$ is analogously. Let $f_1$ part of the non-negative mild solution to \eqref{BGK}-\eqref{temp}. Then it satisfies by definition
{\scriptsize
\begin{align}
\begin{split}
f_1(x,v,t)&= e^{-\alpha_1(x,v,t)} f_1^0(x-tv,v) \\ &+ e^{-\alpha_1(x,v,t)} \int_0^t [ \tilde{\nu}_{11} \frac{n_1(x+(s-t)v,s)}{n_1(x+(s-t)v,s)+ n_2(x+(s-t)v,s)} M_1(x+(s-t)v,v,s) \\ &+ \tilde{\nu}_{12} \frac{n_2(x+(s-t)v,s)}{n_1(x+(s-t)v,s)+ n_2(x+(s-t)v,s)} M_{12}(x+(s-t)v,v,s)]] e^{\alpha_1(x+(s-t)v,v,s)} ds.
\end{split}
\label{mild}
\end{align}}
We assumed that all collision frequencies are positive and according to lemma \ref{est} all densities are positive. So the right-hand side of \eqref{mild} is positive, therefore
$$ f_1(x,v,t) > 0,$$
for positive initial data.
So non-negative solutions to \eqref{BGK} - \eqref{temp} are even positive.
$\hfill\Box$
\end{proof}

\subsection{Positivity of solutions}
\label{sec4.4}
\begin{thm}
Let $(f_1,f_2)$ be a classical solution to \eqref{BGK} - \eqref{temp} with positive initial data. Then the solution is positive meaning $f_1, f_2 >0$. 
\end{thm}
\begin{proof}
According to Theorem \ref{ex} there exists a non-negative solution to \eqref{BGK} - \eqref{temp} and it is the only  non-negative solution to \eqref{BGK} - \eqref{temp}. So there could exist another classical solution which at a certain time becomes zero and negative afterwards. But due to continuity in time, it could only happen if it reaches zero first. According to Lemma \ref{est} this is not possible, because a non-negative solution always stays positive. So the unique solution to \eqref{BGK} - \eqref{temp} with positive initial data is positive meaning $f_1,f_2 > 0$.
$\hfill\Box$
\end{proof}

\section{Determination of an unknown function in the energy exchange of Dellacherie}
\label{sec5}
This final section will show the usefulness of our kinetic description  in a macroscopic model by Dellacherie \cite{Dellacherie}.
In particular in this section\textcolor{black}{,} we choose the space dimension $N$ equal to $3$ and want to use the model described in subsection \ref{sec2.1} in order to determine an unknown function in the energy exchange in the macroscopic model of Dellacherie \cite{Dellacherie}. In subsection \ref{sec5.1}\textcolor{black}{,} we introduce the macroscopic model of Dellacherie and compare the moment equations of our kinetic model in subsection\ref{sec2.1} with the model of Dellacherie in order to determine his unknown function in the energy exchange.
\subsection{Macroscopic Model of Dellacherie}
\label{sec5.1}
We consider the macroscopic model for \textcolor{black}{a two component} gas mixture from the literature \cite{Dellacherie}. Each gas consisting of particles of the mass $m_k$ is characterized by a density $n_k$, a mean velocity $u_k$ and an energy $E_k$, $k=1,2.$ 
Dellacherie in \cite{Dellacherie} proposes a macroscopic model for gas mixtures given by
{\footnotesize
\begin{align}
\begin{split}
\partial_t \begin{pmatrix}
 m_1 n_1 \\ m_2 n_2 \\ m_1 n_1 u_1 \\ m_2 n_2 u_2 \\ m_1 n_1 E_1 \\ m_2 n_2 E_2
\end{pmatrix} + \nabla_x \cdot \begin{pmatrix}
 m_1 n_1 u_1 \\ m_2 n_2 u_2 \\ m_1 n_1 u_1 \otimes u_1 +  p_1 \textbf{1} \\ m_2 n_2 u_2 \otimes u_2 +  p_2 \textbf{1} \\  u_1 ( m_1 n_1 E_1 + p_1) \\  u_2 (m_2 n_2 E_2 + p_2)
\end{pmatrix}  = \begin{pmatrix}
0 \\0 \\ \lambda_u (u_2 - u_1) \\ \lambda_u (u_1-u_2) \\ \lambda_T (T_2-T_1) + \lambda_u U(u_1,u_2) \cdot (u_2-u_1) \\ \lambda_T (T_1- T_2) + \lambda_u U(u_1,u_2) \cdot (u_1 - u_2),
\end{pmatrix}
\end{split}
\label{Dell}
\end{align}}
where  $U(u_1,u_2)$ is an unknown function of the velocities $u_1,u_2$ and $\lambda_u, \lambda_T$ are relaxation parameters determined by physical experiments.
The temperature $T_k$ and the pressure $p_k$ are related by the equation of an ideal gas given by $p_k = n_k T_k.$ The unknown function $U$ is inside the relaxation term in the energy equations which forces the gas mixture to go to a common velocity in thermodynamic equilibrium. Dellacherie \cite{Dellacherie} has the following restriction on $U$. He can show that his macroscopic model for gas mixtures satisfies an H-Theorem as soon as $U$ verifies the condition
\begin{align}
\min(u_1,u_2) \leq U(u_1, u_2) \leq \max(u_1,u_2) .
\label{restric}
\end{align}
 With this restriction on $U$ in \eqref{restric} Dellacherie is able to prove that for $\lambda_u, \lambda_T \rightarrow 0 $ the model converges formally to a macroscopic model for the densities, the total momentum and the total energy. 
\subsection{Comparison with the energy exchange terms obtained from the BGK model for mixtures}
\label{sec5.2}



Now, our aim is to derive a macroscopic equation for the energy of the kinetic BGK model \eqref{BGK} and to determine the parameter $\gamma$ in the definition of the mixture temperature $T_{12}$ in \eqref{contemp}. 
\begin{lem}
Assume \eqref{coll}, the conditions \eqref{density}, \eqref{convexvel} and \eqref{contemp}. Then the momentum and energy exchange term of species $1$ of the model \eqref{BGK} are given by 
\begin{align}
F_{m_{1,2}} &= m_1 \nu_{12} n_1 n_2 (1-\delta) (u_2 - u_1), \\
\begin{split}
F_{E_{1,2}}&= [\nu_{12} \frac{1}{2} n_1 n_2 m_1 (\delta -1) (u_1 + u_2 + \delta(u_1-u_2)) \\&+ \frac{3}{2} \nu_{12} n_1 n_2 \gamma (u_1 - u_2)] \cdot (u_1-u_2) \quad + \frac{3}{2} \nu_{12} n_1 n_2 (1-\alpha) (T_1-T_2).
\end{split}
\label{en}
\end{align}
\end{lem}
The momentum and energy exchange terms of species $1$ are obtained by multiplying the right-hand side of the first equation of \eqref{BGK} by $v$ and $|v|^2$, respectively and integrating the result with respect to $v$, for more details see the proof of Theorem 2.3 in \cite{Pirner}. We will get the following relationship between the energy exchange of the two models \eqref{BGK} and \eqref{Dell}. 
\begin{thm}
Assume $\delta<1$.
The two energy exchange terms \eqref{en} and the one in \eqref{Dell} coincide if $U$ is of the form 
$$ U(u_1,u_2) = \frac{1}{2} \frac{(u_1+u_2) \cdot (u_1-u_2)}{|u_1-u_2|^2} (u_1-u_2) + c(u_1-u_2) + V_{\perp} (u_1,u_2), \quad c \in \mathbb{R},$$
where $V_{\perp}$ is a function parallel to $u_1-u_2$.
\end{thm}
\begin{proof}
In order to have equality with the exchange term from Dellacherie, we want that
\begin{align*}
F^{vel}_{E_{1,2}} &:= [\nu_{12} \frac{1}{2} n_1 n_2 m_1 (\delta -1) (u_1 + u_2 + \delta(u_1-u_2))\\& + \frac{3}{2} \nu_{12} n_1 n_2 \gamma (u_1 - u_2)] \cdot (u_1-u_2) \\ &\stackrel{!}{=} - \lambda_u U(u_1,u_2) \cdot (u_1 - u_2),
\end{align*}
which is equivalent to 

\begin{align}
\begin{split}
[ \frac{1}{2} \textcolor{black}{\nu_{12} n_1 n_2} \ m_1 (\delta -1) (u_1 + u_2 + \delta(u_1-u_2)) + \frac{3}{2} \nu_{12} n_1 n_2 \gamma (u_1 - u_2)+ \lambda_u U(u_1,u_2) ]\\ \cdot (u_1-u_2) = 0.
\label{2}
\end{split}
\end{align}
This means that $$[\nu_{12} \frac{1}{2} n_1 n_2 m_1 (\delta -1) (u_1 + u_2 + \delta(u_1-u_2)) + \frac{3}{2} \nu_{12} n_1 n_2 \gamma (u_1 - u_2)+ \lambda_u U(u_1,u_2) ]$$ has to be orthogonal to $u_1-u_2$. \\
We split all terms in a term parallel and a term orthogonal to $u_1-u_2$:
\begin{align*}
U(u_1,u_2)&= v(u_1,u_2)  (u_1-u_2) + V_{\perp}(u_1,u_2), \\
u_1+u_2 &= \left[(u_1+u_2) \cdot \frac{(u_1-u_2)}{|u_1-u_2|} \right]\frac{u_1-u_2}{|u_1-u_2|} + u_{\perp} (u_1,u_2).
\end{align*} 
Now the fact that the whole expression has to be orthogonal to $u_1-u_2$ means that the sum of coefficients in front of $u_1-u_2$ in \eqref{2} has to vanish. This leads to
\begin{align}
[\nu_{12} \frac{1}{2} n_1 n_2 m_1 (\delta -1) (\frac{(u_1+u_2)\cdot (u_1-u_2)}{|u_1-u_2|^2} + \delta) + \frac{3}{2} \nu_{12} n_1 n_2 \gamma + \lambda_u v(u_1,u_2) ]=0.
\label{1}
\end{align}
In order to get equality in the exchange of momentum, we have to choose 
\begin{align}
\delta = 1- \frac{\lambda_u}{m_1 \nu_{12} n_1 n_2}.
\label{delta}
\end{align}
If we use this expression for $\delta$ given by \eqref{delta} and solve \eqref{1} for $\gamma$, we obtain
\begin{align}
\begin{split}
\gamma = \frac{1}{3} m_1 (1-\delta) \frac{(u_1+u_2)\cdot (u_1-u_2)}{|u_1-u_2|^2} + \frac{1}{3} m_1 (1- \delta) \delta - \frac{2}{3} \lambda_u v(u_1,u_2) \frac{1}{n_1 n_2 \nu_{12}} \\ =\frac{1}{3} m_1 (1-\delta) \frac{(u_1+u_2)\cdot (u_1-u_2)}{|u_1-u_2|^2} + \frac{1}{3} m_1 (1- \delta) \delta - \frac{2}{3} m_1 (1- \delta) v(u_1,u_2). 
\end{split}
\label{U}
\end{align}
Since we assumed $\gamma$ to be a parameter independent of the velocities, we deduce 
$$ v(u_1,u_2) = \frac{1}{2} \frac{(u_1 + u_2) \cdot (u_1-u_2)}{|u_1 - u_2|^2 }- c, \quad c \in \mathbb{R}, $$
\textcolor{black}{for $\delta <1$.}
$\hfill\Box$
\end{proof}
For $\gamma$ this leads to
\begin{align}
 \gamma = \frac{1}{3} m_1 (1- \delta) \delta +\frac{2}{3} m_1 (1- \delta) c. 
 \label{gam}
 \end{align}
We also get a restriction on $U$ like Dellacherie. 
\begin{lem}[Restriction on the constant $c$]
If we assume that all temperatures are positive, we get the following restriction on the constant $c$ given by
\begin{align}
- \frac{1}{2} \delta \leq c \leq - \frac{1}{2} ( \frac{m_1}{m_2} \varepsilon (1- \delta) ) + \frac{1}{2}.
\label{c}
\end{align}
\end{lem}
\begin{proof}
In order to have positive temperatures in the two species BGK model, we need that $\gamma$ satisfies the condition \eqref{gamma}. 

We see from \eqref{delta} that  $\delta \leq 1$, since $\lambda_u, m_1, \nu_{12}, n_1, n_2$ are assumed to be positive. This leads to the restriction on the constant $c$ given by \eqref{c}.
$\hfill\Box$
\end{proof}
$\gamma$ is a non-negative number, so the right-hand side of the inequality in \eqref{gamma} must be non-negative. This condition is equivalent to \eqref{gammapos}.
With this restriction on $\delta$\textcolor{black}{,} we can deduce from \eqref{c} the estimate
\begin{align*}
- \frac{1}{2} \leq c\leq \frac{1}{2}.
\end{align*}
This corresponds to the estimate \eqref{restric} on $U$ from \cite{Dellacherie}. With \eqref{c}\textcolor{black}{,} we have a more restrictive estimate on the function $U$ and with \eqref{U} an explicit expression of the parallel part of $U$. The orthogonal part does not matter because it does not enter in the exchange term.
\subsection{Determine c by symmetry arguments}
\label{sec5.3}
In the kinetic model in \cite{Pirner} the mixture temperature $T_{12}$ of species $1$ is given by 
\eqref{contemp}
and the one of species $2$ by
\eqref{temp}.
Due to symmetry arguments\textcolor{black}{,} we choose the term in front of $|u_1 - u_2|^2$ in the temperature $T_{21}$ such that it is equal to $\varepsilon \gamma = \frac{1}{3} \varepsilon m_1 (1- \delta) \delta + \varepsilon \frac{2}{3} m_1 (1- \delta) c$ using $\gamma$ given by \eqref{gam}. Comparing the coefficient in front of $|u_1-u_2|^2$ with this expression for $\varepsilon \gamma$  leads to a value for the constant $c$ given by
$$ c= \frac{1}{4} (1- \delta) (1- \frac{m_1}{m_2} \varepsilon ).$$
It remains to show that this specific $c$ satisfies the estimates \eqref{c}. First, the estimate from below. If we use \eqref{gammapos}\textcolor{black}{,} we obtain 
$$ c= - \frac{1}{4} (\frac{m_1}{m_2} \varepsilon - 1)(1- \delta) \geq - \frac{1}{4} \delta (1+ \frac{m_1}{m_2} \varepsilon ) (1- \delta).$$
Rearranging \eqref{gammapos} to 
\begin{align}
 1- \delta \leq \frac{2}{1+ \frac{m_1}{m_2} \varepsilon},
 \label{1-delta}
 \end{align}
 leads to
 $$c \geq - \frac{1}{2} \delta.$$
 The estimate on this specific $c$ from above is equivalent to
 $$\frac{1}{4} \frac{m_1}{m_2} \varepsilon (1- \delta) + \frac{1}{4} (1- \delta) \leq \frac{1}{2}.$$
 By using \eqref{1-delta}\textcolor{black}{,} we get
 $$ \frac{1}{4} (\frac{m_1}{m_2} \varepsilon + 1) (1- \delta) \leq \frac{1}{4} \frac{2}{1- \delta} (1- \delta) = \frac{1}{2}.$$
 In summary\textcolor{black}{,} we are able to determine more accurately the energy exchange in a model by Dellacherie.
\section*{Acknowledgements}
This research was supported by the German Priority Programme 1648. We  thank Gabriella Puppo for many discussions on multi-species kinetics.



\section*{References:}
\begin{thebibliography}{99}

\bibitem{AndriesAokiPerthame2002}
 P. Andries, K. Aoki and B. Perthame,\emph{ A consistent BGK-type model for gas mixtures,} Journal of Statistical Physics 106 (2002) 993-1018
 
 \bibitem{Andries} P. Andries, J. Bourgat, P. Le Tallec, B. Perthame, \emph{Numerical comparison between the Boltzmann and ES-BGK models for rarefied gases}, research report, HAL (2006) 813-830

 \bibitem{Bennoune_2008}
M. Bennoune, M. Lemou and L. Mieussens, 
\emph{ Uniformly stable numerical schemes for the Boltzmann equation preserving the compressible Navier-Stokes asymptotics,} Journal of Computational Physics 227 (2008) 3781-3803

\bibitem{Bernard_2015}
F. Bernard, A. Iollo and G. Puppo, 
\emph{ Accurate asymptotic preserving boundary conditions for kinetic equations on Cartesian grids,} Journal of Scientific Computing  65 (2015) 735-766
 
 \bibitem{Bisi}
M.Bisi, M. C\'aceres, \emph{A BGK relaxation model for polyatomic gas mixtures,} Communication in Mathematical Sciences, 14 (2016) 297-325
 
 \bibitem{Brull_2012}
 S. Brull, V. Pavan and J. Schneider, {\sl Derivation of a BGK model for mixtures}, European Journal of Mechanics B/Fluids, 33 (2012) 74-86

 
 \bibitem{Cercignani}
 C. Cercignani,  Rarefied Gas Dynamics, From Basic Concepts to Actual Calculations, Cambridge University Press, 2000 
 
  \bibitem{Cercignani_1975}
 C. Cercignani, The Boltzmann Equation and its Applications, Springer, 1975
 
 \bibitem{Crestetto_2012}
A. Crestetto, N. Crouseilles and M. Lemou,
\emph{ Kinetic/fluid micro-macro numerical schemes for Vlasov-Poisson-BGK equation using particles,}  Kinetic and Related Models  5 (2012) 787-816
 
 \bibitem{Pirner4}
A. Crestetto, C. Klingenberg, M. Pirner, \emph{Kinetic/fluid micro-macro numerical scheme for a two component plasma}, submitted, 2017

\bibitem{Dellacherie} S. Dellacherie, \emph{ Relaxation Schemes For The Multicomponent Euler System,  Mathematical Modelling and Numerical Analysis} 37 (2003) 909-936

 \bibitem{Dimarco}
G. Dimarco, L. Mieussens and V. Rispoli, 
{\sl An asymptotic preserving automatic domain decomposition method for the Vlasov-Poisson-BGK system with applications to plasmas, } Journal of Computational Physics 274 (2014) 122-139

\bibitem{Dimarco_2014}
G. Dimarco and L. Pareschi, 
{\sl Numerical methods for kinetic equations,} 
Acta Numerica 23 (2014) 369-520
\bibitem{Garzo1989} 
V. Garz{\'o}, A. Santos and J. J. Brey, {\sl A kinetic model for a multicomponent gas} Physics of Fluids, 1 (1989) 380-383

 
 \bibitem{gross_krook1956} E. P. Gross and M. Krook, {\sl Model for collision processes in gases: small-amplitude oscillations of charged two-component systems,} Physical Review 3 (1956) 593
 
 \bibitem{hamel1965} B. Hamel, {\sl Kinetic model for binary gas mixtures}, Physics of Fluids 8 (1965) 418-425
 
\bibitem{Jin_2010}
F. Filbet and S. Jin,
{\sl A class of asymptotic-preserving schemes for kinetic equations and related problems with stiff sources,} Journal of Computational Physics 20 (2010) 7625-7648

 \bibitem{Perthame} B. Perthame, M. Pulvirenti, \emph{Weighted $L^{\infty}$ Bounds and Uniqueness for the Boltzmann BGK Model,} Arch. Rational Mech. Anal. 125 (1993) 289-295
 
\bibitem{Pirner5} C. Klingenberg, M. Pirner, \emph{Chapman-Enskog expansion for gas mixtures}, submitted, 2017 
 
 \bibitem{Pirner} C. Klingenberg, M.Pirner, G.Puppo, \emph{A consistent kinetic model for a two-component mixture with an application to plasma}, Kinetic and related Models 10 (2017) 445-465
 
    \bibitem{Pirner2} C. Klingenberg, M. Pirner, G. Puppo, \emph{Kinetic ES-BGK models for a multi-component gas mixture}, Springer Proceedings in Mathematics and Statistics of the International Conference on Hyperbolic Problems: Theory, Numeric and Applications in Aachen 2016, (2017) 
  
 \bibitem{Puppo_2007}
S. Pieraccini and G. Puppo, 
{\sl Implicit-explicit schemes for BGK kinetic equations}, Journal of Scientific Computing 32 (2007) 1-28

\bibitem{Sofonea2001} V. Sofonea and R. Sekerka, {\sl BGK models for diffusion in isothermal binary fluid systems,} Physica, 3 (2001), 494-520

   \bibitem{Yun} S.Yun, \emph{Classical solutions for the ellipsoidal BGK model with fixed collision frequency}, Journal of Differential Equations 259 (2015) 6009 - 6037

\bibitem{Yun2007} S. Ha, S. Noh, S. Yun, \emph{Global existence and stability of mild solutions to the Boltzmann system for gas mixtures,} Quarterly of Applied Mathematics 4 (2007) 757-779

\end{thebibliography}

\end{document}